\documentclass[12pt]{article}
\usepackage{amsfonts}
\usepackage{amsmath}
\usepackage{latexsym}
\usepackage{amscd}

\newtheorem{theorem}{Theorem}
\newtheorem{lemma}{Lemma}
\newtheorem{definition}{Definition}

\newtheorem{corollary}{Corollary}
\newtheorem{example}{Example}

\def\R{\mathbb{R}}

\def\sup{\mathop{\rm sup }}

\def\Z{\rm  Z }

\newcommand{\be}{\begin{equation}}
\newcommand{\ee}{\end{equation}}

\begin{document}
\bibliographystyle{plain}

\thispagestyle{empty}
\setcounter{page}{0}

\vspace {2cm}

{\Large G. Morvai and  B.  Weiss: }

\vspace {2cm}

{\Large Intermittent estimation of stationary time series.}

\vspace {2cm}

{\Large   Test  13  (2004),  no. 2, 525--542.  }

\vspace {2cm}

\begin{abstract}
Let  $\{X_n\}_{n=0}^{\infty}$ be a stationary  real-valued time series with unknown distribution. 
Our goal is to   
estimate the conditional expectation of $X_{n+1}$ based on the observations 
$X_i$, $0\le i\le n$ in a strongly consistent way.   
Bailey and Ryabko proved that this is not possible even for ergodic binary time 
series if one estimates at all values of $n$. 
 We propose a very simple algorithm  which will  make  prediction 
 infinitely often at carefully selected stopping times chosen 
by our rule. We show that under certain conditions 
our procedure is strongly (pointwise) consistent, and $L_2$ consistent without any condition. 
 An upper bound on the growth  
of the stopping times is also presented in this paper.  
\end{abstract}

\pagebreak

\section{Introduction}

Let  $\{X_n\}_{n=0}^{\infty}$ be a real-valued time series. 
We are interested in estimating the random variable $X_{n+1}$ 
given the past observations $X_0,\dots, X_n$. 
If the random variable $X_{n+1}$ has finite expectation 
and  we are to mimimize the conditional mean squared error
then the solution 
is to choose the conditional expectation 
$E(X_{n+1}|X_0,\dots,X_n)$ . 
Usually, the distribution is not known a priori. In this case we may try to 
estimate the above quantity from observations.

Assume the distribution of the real-valued time series $\{X_n\}_{n=0}^{\infty}$ 
is stationary. 
Now the goal is to estimate 
the conditional expectation $E(X_{n+1}|X_0,\dots,X_n)$ from the data segment 
$X_0,\dots, X_n$ such that 
the difference between the estimate and 
the conditional expectation should tend to zero almost surely as 
the number   of observations  $n$  
tends to infinity. However
   \cite{GYMY98} proved that there is no such 
estimator if one estimates for all values of $n$,  even for   
all stationary and ergodic first order Markov chains  
taking values from the unit interval 
$[0,1]$. This problem was posed originally 
in  \cite{Cover75}.    \cite{Bailey76}  
(applying the method of cutting and stacking developped in  \cite{Ornstein74} 
and  \cite{Shields91}) 
constructed a family of sationary and ergodic binary processes 
such that for any  estimation scheme there was  a process 
in his family for which the difference between 
the estimate and the true conditional expectation did not  tend to zero. 
(Cf.  \cite{Ryabko88} also.)

However, for the class of all stationary and ergodic binary Markov chains of some finite 
 order one can solve this problem.
 Indeed,  
if the time series is a Markov chain of some finite (but unknown) order, we can estimate the
 order 
(cf.
\cite{CsSh00}, and  \cite{Csiszar02}) and  count frequencies of  blocks with length equal to the order.

In another  special case, for certain Gaussian processes,    \cite{Schafer} 
constructed an estimator such that 
for that family of processes the error between his estimator and the true 
conditional expectation tends to zero almost surely 
as the number of observations increases. 
 
Here we note that a totally different problem is when the goal is 
to estimate the conditional expectation 
in such a way that the {\it time average} of the 
squared error is required to vanish as the number of observations tends to infinity. 
This problem can be easily solved, cf. \cite{Bailey76},  \cite{Ornstein78}, 
 \cite{Algoet92, Algoet94, Algoet99}, 
 \cite{MoYaAl97},   
\cite{MoYaGy96},  \cite{GYLM99} and 
\cite{GYL02}.  
(See also  \cite{Weiss00} and  \cite{GYKKW02}.)

In this paper the setting is  different. We do not weaken the error criterion, 
that is, 
we will further consider the difference between our estimate and the 
true conditional expectation 
(rather than time averages) but we do not require to estimate for every time instance $n$, 
but rather, 
merely along  a stopping time sequence. That is, looking at the data segment 
$X_0,\dots,X_n$ our rule 
will decide if we dare to estimate for this $n$ or not, 
but anyhow we will definitely estimate 
for infinitely many $n$. 

Such algorithm was proposed for binary time series in \cite{Morvai00} but 
there the growth of the stopping times is like an exponential tower, 
and so that scheme is not feasible at all. A more practical algorithm was proposed in 
 \cite{MW02}
 for certain binary time series. 
In this paper we provide an  algorithm for  real-valued processes.

\section{Definition of the Estimator and Main Results}

For some technical reason, we will consider 
two-sided stationary real-valued processes $\{X_n\}_{n=-\infty}^{\infty}$. 
Note that a one-sided  stationary time series $\{ X_n\}_{n=0}^{\infty}$ 
can be extended to be a two-sided stationary time series 
$\{ X_n\}_{n=-\infty}^{\infty}$. 

\bigskip
\noindent
For notational convenience, let 
$X_m^n=(X_m,\dots,X_n)$,
where $m\le n$. Let $\{ {\cal P}_k\}_{k=0}^{\infty}$ denote a nested sequence of finite or countably infinite partitions of the real line by
intervals.   
Let $x\rightarrow [x]^k$ denote a quantizer that assigns to any point $x\in \R$ the  unique interval in ${\cal P}_k$ that contains $x$. 
For a set $C\subseteq \R $ let ${\rm diam}(C)=\sup_{y,z\in C} |z-y|$. 
We assume that 
\begin{equation}
\label{partitionassumption}
\lim_{k\to \infty} {\rm diam}([x]^k)=0 \ \ \mbox{for all $x\in \R$.}
\end{equation} 
Let $[X_m^n]^k=([X_m]^k,\dots,[X_n]^k)$. Let $1\le l_k\le k$ be a nondecreasing sequence of 
positive  integers such that $\lim_{k\to\infty} l_k=\infty$. Put  
$$
J(n)=\min\{ j\ge 1: \ \ l_{j+1}>n\}.
$$

\bigskip
\noindent
Define the stopping times  as follows.  Set $\zeta_0=0$. 
For $k=1,2,\ldots$,  define the sequences $\eta _k$ and $\zeta_k$
recursively. 
Define 
$$
{\eta}_1=
\min\{t>0 : [X_{\zeta_{0}-(l_1-1)+t}^{\zeta_{0}+t}]^1=[X_{\zeta_{0}-(l_1-1)}^{\zeta_{0}}]^1\}
\ \ \mbox{and} \ \ 
\zeta_1=\eta_1.
$$
Next we refine the quantization and look for the next occurrence of the block of length $l_2$, namely 
$$
{\eta}_2=
\min\{t>0 : [X_{\zeta_{1}-(l_2-1)+t}^{\zeta_{1}+t}]^2=[X_{\zeta_{1}-(l_2-1)}^{\zeta_{1}}]^2\}
\ \ \mbox{and} \ \ 
\zeta_2=\zeta_{1}+\eta_2.
$$
In general, we refine the quantization, and slowly increase the block length of the next repetition, as follows: 
\begin{equation}
\label{defstoppingtime}
{\eta}_k=
\min\{t>0 : [X_{\zeta_{k-1}-(l_k-1)+t}^{\zeta_{k-1}+t}]^k=
[X_{\zeta_{k-1}-(l_k-1)}^{\zeta_{k-1}}]^k\}
\ \ \mbox{and} \ \ 
\zeta_k=\zeta_{k-1}+\eta_k.
\end{equation}
One denotes the $k$th estimate of $E(X_{\zeta_k+1}|X_0^{\zeta_k})$ by $g_k$, 
and defines it to be
\begin{equation}
\label{fkdistrestimate2}
g_k=
{1\over k}\sum_{j=0}^{k-1} X_{\zeta_j+1}. 
\end{equation}

\bigskip
\noindent Let $\R$ be the set of all real numbers and put ${\R}^{*-}$  the set of all one-sided 
  sequences of real numbers, 
that is, 
$${\R}^{*-} =\{ (\dots,x_{-1},x_0): x_i\in \R \ \ 
\mbox{for all $-\infty<i\le 0$}\}.$$
Define a metric on sequences
 $(\dots,x_{-1},x_{0},)$ and $(\dots,y_{-1},y_{0})$  as follows. Let 
\begin{equation}
\label{defdistance}
d^*((\dots,x_{-1},x_{0}),(\dots,y_{-1},y_{0}))=
\sum_{i=0}^{\infty} 2^{-i-1} {
|x_{-i}-y_{-i}|\over 1+|x_{-i}-y_{-i}| }.
\end{equation} 
(For details see \cite{Gr88} p. 51. )

\bigskip
\noindent
\begin{definition} [Almost surely continuous conditional expectation.]
The conditional expectation  $E(X_1|\dots,X_{-1},X_{0})$
is almost surely continuous if for some  set $C\subseteq {\R}^{*-}$  
which has probability one   
the conditional expectation $E(X_1|\dots,X_{-1},X_{0})$ restricted to this  set $C$ 
is continuous with respect to metric $ d^*(\cdot,\cdot)$ in (\ref{defdistance}). 
\end{definition}

\begin{example} A stationary and ergodic time series with  almost  surely continuous 
conditional expectation 
which is not continuous on the whole space.

\bigskip
We will define a transformation $S$ on the unit interval.
 Consider the binary expansion 
$r_1^{\infty}$  of each real-number $r\in [0,1)$, that is, 
$r=\sum_{i=1}^{\infty} r_i 2^{-i}$.  When there are two expansions,
use the representation which contains finitely many  $1's$.
Now let 
$$
\tau(r)= \min\{i>0: r_i=1\}.
$$
Notice that, aside from the exceptional set  $\{0\}$,  which has Lebesgue measure zero $\tau$  is finite and 
well-defined on the closed unit interval.   
The transformation is defined by 
\begin{equation}
(Sr)_i=\left\{
\begin{array}{ll}
1 & \mbox{if $0<i<\tau(r)$} \\
0 & \mbox{if $i=\tau(r)$} \\
r_i & \mbox{if $i>\tau(r)$}.
\end{array}
\right. 
\end{equation}
Notice that in fact, $Sr=r-2^{-\tau(r)}+\sum_{l=1}^{\tau(r)-1} 2^{-l}$.
  All iterations $S^k$ of $S$ for $-\infty<k<\infty$ are well defined and invertible with 
  the exception of the set of dyadic rationals which has Lebesgue measure zero. 
 This transformation $S$ could be  defined recursively as 
\begin{equation}
Sr=\left\{
\begin{array}{ll}
r-0.5 & \mbox{if $0.5\le r<1$} \\
{1+S(2r) \over 2} & \mbox{if $0\le r< 0.5$.} 
\end{array}
\right.
\end{equation}
Now choose $r$ uniformly on the unit interval. Set $X_0(r)=r$  and put $X_n(r)=S^n r$. 
Notice that the resulting 
time series 
$\{X_n\}$  is a stationary and ergodic Markov chain with order one, 
cf.  \cite{GYMY98}.
What more, one observation
 determines the whole orbit of the process. 
Observe that $E(X_{n+1}|X_0^n)=E(X_{n+1}|X_n)$ and $E(X_{n+1}|X_n=x)=S x$. Since $S$ is a continuous mapping 
disregarding the set of dyadic rationals, 
the resulting conditional expectation is almost surely continuous. However, the conditional 
expectation 
is not continuous on the whole unit interval, since it can not be made continuous, 
for example, at $0.5$. 
\end{example}

\bigskip
\noindent
The next theorem establishes the strong (pointwise) consistency of the  proposed estimator.

\begin{theorem} \label{Theorem3} Let  $\{X_n\}$ be a real-valued stationary time series 
with $E(|X_0|^2)<\infty$. 
For the estimator $g_k$ defined in~(\ref{fkdistrestimate2}) 
and for the stopping time $\zeta_k$ defined in~(\ref{defstoppingtime}),
$$
\lim_{k\to\infty} \left| g_k- E(X_{\zeta_k+1}|X_0^{\zeta_k})\right| =0\ \ 
\mbox{almost surely}
$$
provided that the conditional expectation  $E(X_1|X_{-\infty}^{0})$ is almost surely 
continuous. 
\end{theorem}

\bigskip
The proof of  Theorem~\ref{Theorem3} involves both the martingale convergence
theorem and classical convergence results for an auxilliary sequence
of orthogonal random variables. The assumption on the 
almost sure continuity of the conditional expectation
is crucial in going from the auxilliary variables to the actual 
random variables that take part in the estimator.

\bigskip
 The consistency  holds independently of how the  sequence $l_k$ and
the partitions are chosen as long as $l_k$ goes to infinity and the
partitions become finer.  However, the choice of these sequences has a
great influence on the growth of the stopping times.

\bigskip
From the proof of  \cite{Bailey76}, \cite{Ryabko88} and 
\cite{GYMY98} it is clear that even for  
the class of all stationary and ergodic  binary 
time series with almost surely continuous conditional expectation  
$E(X_1|\dots,X_{-1},X_{0})$ one can not estimate $E(X_{n+1}|X_0^n)$ for all $n$ 
strongly (pointwise) 
consistently.

Note that the processes constructed by the method of 
cutting and stacking 
(cf. \cite{Ornstein74} and \cite{Shields91})
are stationary processes with 
almost surely continuous conditional expectations.

The stationary processes with almost surely continuous 
conditional expectation  generalize the 
processes for which  the  conditional expectation is actually continuous. 
(Cf.  \cite{Ka90} or   \cite{Ke72}.)

\bigskip
If one's goal is to estimate the conditional mean merely in $L_2$ 
then the problem becomes very easy and even for all  time instances  
one can estimate it, cf.  \cite{MoYaAl97}. 
We will prove that our proposed estimator $\{g_n\}$ along the stopping time sequence 
$\{\zeta_n\}$ is not just strongly consistent under the above mentioned 
continuity condition  
  but also 
consistent in $L_2$ without any continuity  condition. 
The point here is that our scheme achieves 
two goals simultanously.   
In this way, if one runs our algorithm he can be sure that  
if the above mentioned continuity condition holds then the algorithm  achieves strong consistency and 
if unfortunately that condition fails to hold then 
even in that case it achieves $L_2$ consistency. Precisely:

\begin{theorem} \label{Theorem5} Let $\{X_n\}$ be a real-valued stationary time 
series with $E(|X_0|^2)<\infty$. For the estimator
defined in~(\ref{fkdistrestimate2}) and 
for the stopping time $\zeta_k$ defined in~(\ref{defstoppingtime}),
\begin{equation}\label{fkstatement5} 
\lim_{k\to\infty} E\left(\left| g_k-
E(X_{\zeta_k+1}|X_0^{\zeta_k})\right|^2\right) =0. 
\end{equation} 
\end{theorem}

\bigskip
\noindent
The next theorem gives an upper bound on the growth of the stopping times $\{\zeta_k\}$ 
in case when finite partitions are used.  

\bigskip
\noindent
\begin{theorem} \label{Theorem4}
Let  $\{X_n\}$ be a stationary  real-valued time series. 
Assume ${\cal P}_k$ is a nested sequence of finite partitions of the real line by
intervals. If for some $\epsilon>0$, 
$\sum_{k=1}^{\infty} (k+1) 2^{-l_k\epsilon}<\infty$   then  
for the stopping time $\zeta_k$ defined in~(\ref{defstoppingtime}),
$$
\zeta_k< |{\cal P}_k|^{l_k} 2^{l_k\epsilon}
$$
eventually almost surely.
\end{theorem}

\bigskip
\begin{example}
One may set $\epsilon=1$,   
$l_k=\lfloor 3  \log_2 k \rfloor$, and 
$|{\cal P}_k|=\lfloor 2^{ f_k}\rfloor$ where 
$f_k$ is an increasing sequence of positive real numbers tending to infinity arbitrary slowly. 
By Theorem~\ref{Theorem4},  
$\zeta_k< k^{3 (1+f_k)}$, which is almost a polynomial growth.  
\end{example}

In case of finite alphabet processes you can achieve a slightly better upper bound 
than in Theorem~\ref{Theorem4}. Indeed, 
let $H$ denote the entropy rate associated with the
stationary and ergodic finite alphabet time series $\{X_n\}$,  
cf.  \cite{CT91}. 
Note that in this case no quantization is needed. Then it is easy to
see, that $\zeta_k<2^{l_k (H+\epsilon)}$ eventualy almost surely 
provided that $(k+1) 2^{-l_k \epsilon}$ is summable.    
(Cf.  \cite{MW02},   \cite{OrWe93}, 
 \cite{MoYaAl97}.)

\bigskip
 If one desires to estimate $X_{\zeta_j+1}$ in $L_2$ sense based on data $X_0,\dots,X_{\zeta_j}$  then the best he can do is to choose 
the conditional expectation 
$$
g^*_j=E(X_{\zeta_{j+1}}|X^{\zeta_j}_0).
$$
Now we show that the conditional
mean squared error 
$
E((X_{\zeta_j+1}-g_j)^2|X^{\zeta_j}_0)
$
with regard to $g_j$ is close to that of the best possible 
$E((X_{\zeta_j+1}-g^*_j)^2|X^{\zeta_j}_0)$ 
for large $j$. 
Indeed, this  is an immediate consequence of 
Theorem~\ref{Theorem3}, Theorem~\ref{Theorem5}, and the fact that   
$
E((X_{\zeta_j+1}-g_j)^2|X^{\zeta_j}_0)-
E((X_{\zeta_j+1}-g^*_j)^2|X^{\zeta_j}_0)=(g_j-g^*_j)^2.
$ 

\bigskip

\begin{corollary} \label{coroolaryone}
Let $\{X_n\}$ be a stationary real-valued time series. 
Assume  $E(|X_0|^2)<\infty$. Then 
\begin{equation}\label{bayes}
\left|E\left((X_{\zeta_j+1}-g_j)^2|X^{\zeta_j}_0\right)-
E\left((X_{\zeta_j+1}-g^*_j)^2|X^{\zeta_j}_0\right)\right| \to 0
\end{equation}
in $L_1$. Moreover, if in addition,  the conditional expectation 
$E(X_{1}|X^{0}_{-\infty})$ is almost surely continuous, then  
(\ref{bayes})  holds almost surely. 
\end{corollary}

\noindent
Note that 
$X_n$ can not be estimated for all $n$ in such a way  that the conditional mean squared error
tend to zero in the pointwise sense  even in case of almost surely continuous conditional expectation. 
(Cf.  \cite{Bailey76}, \cite{Ryabko88}, 
\cite{GYMY98}.)
The main point here is that along a sequence of stopping times one 
can achieve that property.

\section{Auxiliary Results}

\smallskip
\noindent
It will be useful to define other processes 
for $k\ge 0$ $\{ {\hat X}^{(k)}_n\}_{n=-\infty}^{\infty}$ as follows.
Let 
\begin{equation}
\label{defprocesses1}
\hat X^{(k)}_{-n}=X_{\zeta_k-n} \ \ \mbox{for $-\infty <n<\infty$.}  
\end{equation}

\noindent
For an arbitrary real-valued stationary time series $\{Y_n\}$, for $k\ge 0$ let $\hat\zeta^k_0(Y^0_{-\infty})=0$  
and    for all $k\ge 1$ and $1\le i\le k$ define 
$$
{\hat\eta}^k_i(Y^0_{-\infty})=
\min\{t>0 : 
[Y_{\hat\zeta^k_{i-1}-(l_{k-i+1}-1)-t}^{\hat\zeta^k_{i-1}-t}]^{(k-i+1)}
=
[{Y}_{\hat\zeta^k_{i-1}-(l_{k-i+1}-1)}^{\hat\zeta^k_{i-1}}]^{(k-i+1)}\}
$$
and
$$
\hat\zeta^k_i(Y^0_{-\infty})=\hat\zeta^k_{i-1}(Y^0_{-\infty})-\hat\eta^k_i(Y^0_{-\infty}).
$$
When it is obvious on which time series  ${\hat\eta}^k_i(Y^0_{-\infty})$ and  
$\hat\zeta^k_i(Y^0_{-\infty})$
 are evaluated,  
we will use the notation  ${\hat\eta}^k_i$ and  $\hat\zeta^k_i$.
Let $T$ denote the left shift operator,  
that is, $(T x^{\infty}_{-\infty})_i=x_{i+1}$. It is easy to see that if $\zeta_k(x_{-\infty}^{\infty})=l$ then 
${\hat \zeta}_k^k(T^l x_{-\infty}^{\infty})=-l$.

\smallskip
\noindent
We will need the next lemmas for later use. 

\begin{lemma} \label{distreqlemma} 
Let $\{X_n\}_{n=-\infty}^{\infty}$ be a real-valued stationary process. Then 
the  time series $\{{\hat X}^{(k)}_n\}_{n=-\infty}^{\infty}$, 
$\{X_n\}_{n=-\infty}^\infty$ 
have 
identical  distribution, that is, 
for all $k\ge 0$, $n\ge 0$, $m\ge 0$, and Borel set $F\subseteq {\R}^{n+1}$,
$$
P(({\hat X}^{(k)}_{m-n},\dots,{\hat X}^{(k)}_{m})\in F)=P(X^m_{m-n}\in F).
$$
Thus all the time series $\{{\hat X}^{(k)}_n\}_{n=-\infty}^{\infty}$ for $k=0,1,\dots$ 
are stationary.
\end{lemma}
\begin{proof}
Since the time series $\{X_n\}$ is stationary and for all $k\ge 0$, $n\ge 0$, $l\ge 0$, 
$F\subseteq {\R}^{n+1}$,
\begin{equation}\label{shiftequation} 
T^{l} \{X^{\zeta_k+m}_{\zeta_k+m-n}\in F,\zeta_k=l\} =
\{ X^{m}_{m-n}\in F,{\hat \zeta}^k_k(X^0_{-\infty})=-l\},   
\end{equation}
and by the construction in~(\ref{defprocesses1}), we have  
\begin{eqnarray*}
\lefteqn{ P(({\hat X}^{(k)}_{m-n},\dots,{\hat X}^{(k)}_{m})\in F)}\\
&=&
P(X^{\zeta_k+m}_{\zeta_k+m-n}\in F)
= \sum_{l=0}^{\infty} 
P(X^{\zeta_k+m}_{\zeta_k+m-n}\in F,\zeta_k=l)\\
&=& \sum_{l=0}^{\infty} 
P(X^{m}_{m-n}\in F,{\hat \zeta}^k_k(X^0_{-\infty})=-l)= P(X^m_{m-n}\in F).
\end{eqnarray*}
The proof of the  Lemma~\ref{distreqlemma} is complete. 
\end{proof}

\bigskip
\noindent
For a given $n$, the partition cell $[X_{\zeta_j-n}]^j$ is a random set and is varying as $j\to\infty$. However, we 
will prove that eventually it shrinks. 

\begin{lemma} \label{limitexists} 
Let $\{X_n\}_{n=-\infty}^{\infty}$ be a real-valued stationary process. 
Then for all $n\ge 0$,  
 $\lim_{j\to\infty} {\rm diam}([X_{\zeta_j-n}]^j)=0 $
almost surely. 
\end{lemma} 
\begin{proof}
Observe, that by the definition  of stopping times in (\ref{defstoppingtime}), 
for a given $n$, 
$\{[X_{\zeta_j-n}]^j\}_{j=J(n)}^{\infty}$ is a decreasing sequence of intervals.
Now, if for some $j\ge J(n)$, ${\rm diam} ([X_{\zeta_j-n}]^j)<\infty$ then 
 $\lim_{i\to\infty}{\rm diam} ([X_{\zeta_i-n}]^i)=0$. To seee this notice that  if  
  $\lim_{i\to\infty}  {\rm diam} ([X_{\zeta_i-n}]^i)>0$ then  
$\bigcap_{i=J(n)}^{\infty} [X_{\zeta_i-n}]^i
\neq \emptyset$ and let $z$ denote a real number from this set. For this $z$, 
$\lim_{i\to\infty} {\rm diam} ([z]^i)>0$ contradicting our assumption in 
(\ref{partitionassumption}). 
What remains is to prove that 
$$
P( {\rm diam}[X_{\zeta_j-n}]^j=\infty \ \ \mbox{for all $j\ge J(n)$ })=0.
$$ 
Indeed by Lemma~\ref{distreqlemma} and assumption~(\ref{partitionassumption}),  
\begin{eqnarray*}
\lefteqn { P( {\rm diam}([X_{\zeta_j-n}]^j)=\infty \ \ \mbox{for all $j\ge J(n)$}) } \\
&\le& 
\lim_{j\to\infty} 
P( {\rm diam}([X_{\zeta_j-n}]^j)=\infty)
=\lim_{j\to\infty} P( {\rm diam}([{\hat X}_{-n}^{(j)}]^j)=\infty)\\
&=&\lim_{j\to\infty} P( {\rm diam}([X_{-n}]^j)=\infty)
=\lim_{j\to\infty} P( {\rm diam}([X_{1}]^j)=\infty)=0. 
\end{eqnarray*}
The proof of Lemma~\ref{limitexists} is complete. 
\end{proof}

\bigskip
\noindent
Define the time series $\{ {\tilde X}_n\}_{n=-\infty}^{0}$
\begin{equation}
\label{defprocesses2}
{\tilde X}_{-n}=
\lim_{j\to \infty} X_{\zeta_j-n} \ \ \mbox{for $n\ge 0$,}  
\end{equation}
 where the limit exists  
 since $\{[X_{\zeta_j-n}]^j\}_{j=J(n)}^{\infty}$ is a random sequence of nested
intervals and by Lemma~\ref{limitexists} their lengths tend to zero.

\bigskip

\begin{lemma} \label{distildeeqdishat}
Let $\{X_n\}_{n=-\infty}^{\infty}$ be a real-valued stationary process. 
Then the distribution of $\{ {\tilde X}_n\}_{n=-\infty}^{0}$ equals the distribution of $\{ X_n\}_{n=-\infty}^0$.
\end{lemma}
\begin{proof}
By Lemma~\ref{distreqlemma} it is enough to prove that for any $i\ge 0$, for all $j\ge J(i)$, 
$[{\tilde X}_{-i}]^j=[{\hat X}^{(j)}_{-i}]^j$. 
Let $R_k$ be the set of right  end-points of the right open intervals in the k-th partition, that is,  
$$
R_k=\{ b\in \R\ :\ \exists -\infty < a < b \  [a,b)\in {\cal P}_k \ \ 
\mbox{or} \ \ \exists -\infty \le a < b  \   (a,b)\in {\cal P}_k  \}.
$$ 
Similarly, let $L_k$ be the set of left  
end-points of the left open intervals in the k-th partition, that is,  
$$
L_k=\{ b\in \R\ :\ \exists b < a< \infty \   (b,a]\in {\cal P}_k \ \ 
\mbox{or} \ \ \exists b < a\le  \infty \  (b,a)\in {\cal P}_k  \}.
$$ 
If $[{\tilde X}_{-i}]^j=[{\hat X}^{(j)}_{-i}]^j$ fails for some $j\ge J(i)$ then  
this must happen at some end point, that is,  
${\tilde X}_{-i}\in \bigcup_{k=0}^{\infty} R_k$ or 
${\tilde X}_{-i}\in \bigcup_{k=0}^{\infty} L_k$. ( Since the partition sequence is a nested sequence and 
${\tilde X}_{-i}=\lim_{j\to\infty} {\hat X}^{(j)}_{-i}$.)
Therefore we can estimate: 
By (\ref{defprocesses2}), Lemma~\ref{limitexists}, and Lemma~\ref{distreqlemma}, we have 
\begin{eqnarray*}
\lefteqn{ 1-P( {\tilde X}_{-i}\in [\hat X^{(j)}_{-i}]^j  \ \ \mbox{for all $j\ge J(i)$ } )} \\
&\le& 
\sum_{k=J(i)}^{\infty}\sum_{s\in R_k} P({\tilde X_{-i}}=s, {\hat X}^{(j)}_{-i} <{\tilde X_{-i}} \  \ \mbox{for all $j\ge k$})\\
&+&\sum_{k=J(i)}^{\infty}\sum_{s\in L_k} P({\tilde X_{-i}}=s, {\hat X}^{(j)}_{-i} >{\tilde X_{-i}} \  \ \mbox{for all $j\ge k$})\\
&\le& \sum_{k=J(i)}^{\infty}\sum_{s\in R_k} \lim_{j\to\infty} P( s-{\rm diam}([{\hat X}^{(j)}_{-i}]^j)\le {\hat X}^{(j)}_{-i} <s)\\
&+& \sum_{k=J(i)}^{\infty}\sum_{s\in L_k} \lim_{j\to\infty} P( s< {\hat X}^{(j)}_{-i} \le s+{\rm diam}([{\hat X}^{(j)}_{-i}]^j) )\\
&=& \sum_{k=J(i)}^{\infty}\sum_{s\in R_k} \lim_{j\to\infty} P( s-{\rm diam}([X_{-i}]^j)\le { X}_{-i} <s)\\
&+& \sum_{k=J(i)}^{\infty}\sum_{s\in L_k} \lim_{j\to\infty} P( s<  X_{-i} \le s+{\rm diam}([X_{-i}]^j) )\\
&=& \sum_{k=J(i)}^{\infty}\sum_{s\in R_k} \lim_{j\to\infty} P( s-{\rm diam}([ X_{1}]^j)\le  X_{1} <s)\\
&+& \sum_{k=J(i)}^{\infty}\sum_{s\in L_k} \lim_{j\to\infty} P( s< X_{1} \le s+{\rm diam}([X_{1}]^j) )= 0. 
\end{eqnarray*}
The proof of Lemma~\ref{distildeeqdishat} is complete.
\end{proof}

\bigskip
\noindent
Now it is immediate that the time series $\{ {\tilde X}_n\}_{n=-\infty}^{0}$ 
is stationary, since  $\{ X_n\}_{n=-\infty}^0$ is stationary, and it can be 
extended to be a two-sided time series
$\{ {\tilde X}_n\}_{n=-\infty}^{\infty}$. 
We will use this fact only for the purpose of defining the conditional expectation 
$E({\tilde X}_1|{\tilde X}^{0}_{-\infty})$.

\section{Proof of Theorem~\ref{Theorem3}}

\bigskip
\noindent
\begin{proof}
Define the function $e : {\R}^{*-}\rightarrow (-\infty,\infty)$ as 
$$e(x^{0}_{-\infty})=
E(X_1|X^{0}_{-\infty}=x^0_{-\infty}).
$$
Recall~(\ref{fkdistrestimate2}) and consider 
\begin{eqnarray}
\nonumber
g_k &=& {1\over k}\sum_{j=0}^{k-1} \left(  
X_{\zeta_j+1}  -E(X_{\zeta_j+1}|X_{-\infty}^{\zeta_j}) \right)\\ 
&+& \nonumber {1\over k}\sum_{j=0}^{k-1} E(X_{\zeta_j+1}|X_{-\infty}^{\zeta_j})\\
&=& \label{decomposition} {1\over k}\sum_{j=0}^{k-1} \Gamma_j+
{1\over k}\sum_{j=0}^{k-1}  E(X_{\zeta_j+1}|X_{-\infty}^{\zeta_j}).
\end{eqnarray}

\noindent
Consider the first term and  
observe that
 $\{\Gamma_j\}$ is a sequence of orthogonal random variables 
with $E \Gamma_j=0$ and  
$E\left( \Gamma_j^2\right) \le E\left( (X_1)^2\right)<\infty$ since  
$E\left( \Gamma_j^2\right) \le E\left( (X_{\zeta_j+1})^2\right) $
and, by Lemma~\ref{distreqlemma}, $X_{\zeta_j+1}$ has the same distribution as $X_1$. 
Now by Theorem~3.2.2 in  \cite{Revesz68},
$$
{1\over k}\sum_{j=0}^{k-1} \Gamma_j \to 0 \ \ \mbox{almost surely.}
$$
(Alternatively, you can apply Theorem A6 in  \cite{GYKKW02})

\noindent
Now we deal with the second term. 
For  arbitrary $j\ge 0$, by 
the constructions in (\ref{defprocesses1}),(\ref{defprocesses2})      
\begin{equation}\label{firstjbitequal} 
\lim_{j\to\infty} 
d^*({\tilde X}^0_{-\infty},(\dots,{\hat X}^{(j)}_{-1},
{\hat X}^{(j)}_{0}))=0 \ \ \mbox{almost surely.}
\end{equation}
By assumption, the function $e(\cdot)$ is continuous on a set $C\subseteq {\cal \R}^{*-}$ 
with $P(X^0_{-\infty}\in C)=1$. 
By  Lemma~\ref{distreqlemma}
  and  Lemma~\ref{distildeeqdishat},  
\begin{equation} \label{setcwithprobone}
P({\tilde X}^{0}_{-\infty}\in C, (\dots,{\hat X}^{(j)}_{-1},{\hat X}^{(j)}_0)\in C 
\ \ \mbox{for all $j\ge0$})=1.
\end{equation}  

\noindent 
Now by the continuity of $e(\cdot)$ on the set $C$, and by~(\ref{firstjbitequal}) and 
(\ref{setcwithprobone}),  

\begin{equation} \label{convofcondexp}
E(X_{\zeta_j+1}|X^{\zeta_j}_{-\infty})=e(\dots,{\hat  X}^{(j)}_{-1},{\hat  X}^{(j)}_{0})
\to e({\tilde X}^{0}_{-\infty})=
E({\tilde X}_1|{\tilde X}^{0}_{-\infty}).
\end{equation}
Thus $g_k\to E({\tilde X}_1|{\tilde X}^{0}_{-\infty})$ almost surely. 

\noindent
What remains to be proven is  that almost surely, 
$
E(X_{\zeta_j+1}|X_{0}^{\zeta_j})\to 
E({\tilde X}_1|{\tilde X}^{0}_{-\infty})
$. 

\noindent
For any set $A\subseteq \R$ let ${\rm closure}(A)$ denote the smallest closed subset of 
the real line  containing $A$.  
Put 
$$
S_j(X_0^{\zeta_j})=\{ z^0_{-\infty}\in \R^{*-} :  
z_{-l_{j+1}+1}\in {\rm closure}([X_{\zeta_j-l_{j+1}+1}]^j),\dots, 
z_{0}\in {\rm closure}([X_{\zeta_j}]^j) \}.   
$$
By~(\ref{defstoppingtime}), ~(\ref{defprocesses2}) and (\ref{setcwithprobone}), almost surely, for all $j$,  
\begin{equation}\label{alwaysinc}
X_{-\infty}^{\zeta_j}\in S_j(X_0^{\zeta_j})\bigcap C \ \ 
\mbox{and} \ \ 
{\tilde X}^{0}_{-\infty} \in S_j(X_0^{\zeta_j})\bigcap C.
\end{equation}
Put
$$
\Delta_j(X_0^{\zeta_j})=\sup_{y^0_{-\infty}, z^0_{-\infty}\in S_j(X_0^{\zeta_j})\bigcap C}
|e(y^0_{-\infty})-e(z^0_{-\infty})|.
$$
Now since $e(\cdot)$ is continuous at ${\tilde X}^{0}_{-\infty}$ on set $C$ and 
by (\ref{alwaysinc}) and Lemma~\ref{limitexists}, 
\begin{equation}
 \label{Deltatozero}
\lim_{j\to\infty} \Delta_j(X_0^{\zeta_j})=0 \ \ \mbox{almost surely.}
\end{equation}
By (\ref{Deltatozero}) almost surely, 
\begin{eqnarray}
\lefteqn{\nonumber
\limsup_{j\to\infty}\left|E\left( e({\tilde X}^{0}_{-\infty} )|X_0^{\zeta_j}\right)- 
E\left( e(X_{-\infty}^{\zeta_j})|X_0^{\zeta_j}\right)\right|}\\
&\le&\nonumber
\limsup_{j\to\infty} E\left(\left| e({\tilde X}^{0}_{-\infty} )- 
e(X_{-\infty}^{\zeta_j})\right| |X_0^{\zeta_j}\right)\\
&\le& \nonumber \limsup_{j\to\infty} E\left( \Delta_j(X_0^{\zeta_j})|X_0^{\zeta_j}\right)\\
&=&  \label{difftozero} \limsup_{j\to\infty} \Delta_j(X_0^{\zeta_j}) =0.
\end{eqnarray}
Now consider
$$
E\left(X_{\zeta_j}|X_0^{\zeta_j}\right)=
E\left( e({\tilde X}^{0}_{-\infty})|X_0^{\zeta_j} \right)-
\left\{ E\left( e({\tilde X}^{0}_{-\infty})|X_0^{\zeta_j} \right)-
E\left( e(X_{-\infty}^{\zeta_j})|X_0^{\zeta_j}\right)\right\}.
$$
The first term it is a martingale and tends to $e({\tilde X}^{0}_{-\infty})$ 
by Theorem 7.6.2 in \cite{Ash72})
since by Lemma~\ref{distildeeqdishat},  
$E \left| e({\tilde X}^{0}_{-\infty})\right| = 
E \left| e(X^{0}_{-\infty})\right|\le
 E\left| X_1\right| <\infty$, and ${\tilde X}_{-\infty}^0$ is 
measurable with respect to $\sigma(X_0^{\infty})$.  
The second term tends to zero by ~(\ref{difftozero}). 
The proof of Theorem~\ref{Theorem3} is complete.
\end{proof}

\section{Proof of Theorem~\ref{Theorem5}}

\bigskip
\noindent
\begin{proof}
By Jensen's inequality, (\ref{defprocesses1}) and Lemma~\ref{distreqlemma},   
\begin{eqnarray*} 
\lefteqn{
{1\over 4}  E \left(\left| g_k- E(X_{\zeta_k+1}|X_{0}^{\zeta_k})\right|^2\right) }\\ 
&\le&  E \left(\left| {1\over
k}
\sum_{j=0}^{k-1} \left( X_{\zeta_j+1} - E(X_{\zeta_j+1}|X_{-\infty}^{\zeta_j}) \right) 
\right|^2\right) \\ 
&+&  {1\over
k}\sum_{j=0}^{k-1}
E\left( \left| E({\hat X}^{(j)}_{1}|\dots,{\hat X}^{(j)}_{-1},{\hat X}^{(j)}_0)- 
E({\hat X}^{(j)}_{1}|[{\hat
X}^{(j)}_{-(l_{j+1}-1)},\dots,{\hat X}^{(j)}_0]^{j+1})  \right|^2\right)\\ 
&+&  {1\over k}\sum_{j=0}^{k-1} 
E\left(\left| E({\hat X}^{(k)}_{1}|[{\hat
X}^{(k)}_{-(l_{j+1}-1)},\dots,{\hat X}^{(k)}_0]^{j+1})-
E({\hat X}^{(k)}_{1}|\dots,{\hat X}^{(k)}_{-1},{\hat X}^{(k)}_0)\right|^2\right)\\ 
&+& E\left(
\left| E({\hat X}^{(k)}_{1}|\dots,{\hat X}^{(k)}_{-1},{\hat X}^{(k)}_0)-
 E({\hat X}^{(k)}_{1}|{\hat X}^{(k)}_{{\hat\zeta}^k_k},\dots,{\hat
X}^{(k)}_0)  \right|^2\right), 
\end{eqnarray*} 
where ${\hat\zeta}^k_k$ is evaluated on $\{ {\hat X}^{(k)}_n \}_{n=-\infty}^0$. 
The first
term converges to zero since 
$\Phi_j=X_{\zeta_j+1} -E(X_{\zeta_j+1}|X_{-\infty}^{\zeta_j})$ 
is a sequence of orthogonal random variables  with 
$E(|X_{\zeta_j+1}|^2)=E(|X_0|^2)<\infty$, and 
\begin{equation} \label{loneconv}
E\left(\left| {1\over k} 
\sum_{j=0}^{k-1} 
\Phi_j \right|^2\right) 
=  {1\over k^2} \sum_{j=0}^{k-1} E\left(\left| \Phi_j \right|^2\right) 
\le {1\over k^2}\sum_{j=0}^{k-1} 
 E (|X_{\zeta_j+1}|^2)={1\over k} E (|X_1|^2)\to 0.
\end{equation}
Applying~(\ref{defprocesses1}) and Lemma~\ref{distreqlemma}, one can estimate the sum of the last three terms 
by the sum 
\begin{eqnarray*} 
\lefteqn{\limsup_{k\to\infty} {1\over k}\sum_{j=0}^{k-1} 
E\left(\left|E(X_1|X_{-\infty}^0)- E(X_1|[X_{-(l_{j+1}-1)}^0]^{j+1}) 
\right|^2\right)} \\ 
&+&\limsup_{k\to\infty} {1\over k}\sum_{j=0}^{k-1} 
E\left(\left| E(X_1|[X_{-(l_{j+1}-1)}^0]^{j+1})-E(X_1|X_{-\infty}^0)\right|^2\right)\\ 
&+&\limsup_{k\to\infty}
E\left(\left|E(X_1|X_{-\infty}^0)-E(X_1|X_{{\hat\zeta}^k_k}^0)\right|^2\right), 
\end{eqnarray*} 
where ${\hat\zeta}^k_k$ is now evaluated on $\{ X_n \}_{n=-\infty}^0$. 
All of these terms converge to zero since 
$\lim_{j\to\infty}
E(X_1|X_{-j}^0)=E(X_1|X_{-\infty}^0)$ 
and 
$\lim_{j\to\infty} E(X_1|[X_{-(l_{j+1}-1)}^0]^{j+1})=E(X_1|X_{-\infty}^0)$ 
in $L_2$ by the martingale
convergence theorem, cf. Theorem 7.6.10 and Theorem 7.6.2 in \cite{Ash72},  
and thus the limit in fact exists and equals zero. The proof of Theorem~\ref{Theorem5} is
complete. 
\end{proof}

\section{ Proof of Theorem~\ref{Theorem4}}

\noindent
\begin{proof}
Let $ {\R}^{\Z}$ be the set of all two-sided  sequences of real numbers, that is, 
$${\R}^{\Z}= \{ (\dots,x_{-1},x_0,x_1,\dots):
x_i\in \R \ \ \mbox{for all $-\infty<i< \infty$}\}.
$$

\noindent
Let $y^0_{-l_k+1}\in {\cal P}_k^{l_k}.$ 
Define the set $Q_k(y^0_{-l_k+1})$ as follows: 
$$
Q_k(y^0_{-l_k+1})=\{z^{\infty}_{-\infty}\in {\R}^{\Z} :  -{\hat \zeta}^k_k(z^0_{-\infty}) 
\ge |{\cal P}_k|^{l_k} 2^{l_k\epsilon} , [z^0_{-l_k+1}]^k=y^0_{-l_k+1})\}.
$$
We will estimate the probability of $Q_k(y^0_{-l_k+1})$  by means of the ergodic theorem. To do this apply the ergodic decomposition theorem,
cf. \cite{Gr88}, and denote the distribution according to the ergodic mode $\omega$ by $P_{\omega}$.
Let $x^{\infty}_{-\infty}\in {\R}^{\Z}$ be a typical sequence according to  $P_{\omega}$.  
Define $\alpha_0(y^0_{-l_k+1})=0$ 
and for $i\ge 1$ let 
$$
\alpha_i(y^0_{-l_k+1})=\min \{l> \alpha_{i-1}(y^0_{-l_k+1}): T^{-l} x_{-\infty}^{\infty}\in Q_k(y^0_{-l_k+1})\}.
$$
Define also $\beta_0(y^0_{-l_k+1})=0$  
and for $i\ge 1$ let 
$$
\beta_i(y^0_{-l_k+1})=\min \{l> \beta_{i-1}(y^0_{-l_k+1})+|{\cal P}_k|^{l_k} 2^{l_k\epsilon}: 
T^{-l} x_{-\infty}^{\infty}\in Q_k(y^0_{-l_k+1})\}.
$$
Observe that for arbitrary $i>0$, 
$$
\sum_{j=1}^{\infty}   1_{\{\beta_{i-1}(y^0_{-l_k+1}) < 
\alpha_j(y^0_{-l_k+1}) \le \beta_i(y^0_{-l_k+1})\}}\le k+1.
$$
By  Lemma~\ref{distreqlemma} and ergodicity, 
\begin{eqnarray*}
\lefteqn{ 
P_{\omega}((\dots, {\hat X}^{(k)}_{-1}, {\hat X}^{(k)}_0, {\hat X}^{(k)}_{1},\dots )\in Q_k(y^0_{-l_k+1}))=
P_{\omega}( X_{-\infty}^{\infty} \in Q_k(y^0_{-l_k+1}) ) }\\
&=& 
\lim_{t\to\infty} {1\over \beta_t(y^0_{-l_k+1})} \sum_{j=1}^{\infty} 
1_{ \{\alpha_j(y^0_{-l_k+1})\le \beta_t(y^0_{-l_k+1})\} }  \\
&=&  \lim_{t\to\infty}{1\over \beta_t(y^0_{-l_k+1})} 
\sum_{i=1}^{t} \sum_{j=1}^{\infty}   1_{\{\beta_{i-1}(y^0_{-l_k+1}) < 
\alpha_j(y^0_{-l_k+1}) \le \beta_i(y^0_{-l_k+1})\}} \\
&\le&  \lim_{t\to\infty} 
{t (k+1)\over t |{\cal P}_k|^{l_k} 2^{l_k\epsilon} }
= 
{(k+1) \over |{\cal P}_k|^{l_k} 2^{l_k\epsilon} }. 
\end{eqnarray*}
Since the right hand side does not depend on $\omega$, 
the same upper bound applies for the original stationary time series $\{X_n\}$, that is, 
$$   
P((\dots, {\hat X}^{(k)}_{-1}, {\hat X}^{(k)}_0, {\hat X}^{(k)}_{1},\dots )\in Q_k(y^0_{-l_k+1}))\le
{(k+1) \over |{\cal P}_k|^{l_k} 2^{l_k\epsilon} }. 
$$
By the construction in (\ref{defprocesses1})  
$-{\hat \zeta}_k^k(\dots,{\hat X}^{(k)}_{-1},{\hat X}^{(k)}_{0})=\zeta_k(X_0^{\infty})$ 
   we get  
\begin{eqnarray*}
\lefteqn{
P(\zeta_k(X_0^{\infty})\ge |{\cal P}_k|^{l_k} 2^{l_k \epsilon} )}\\
&=&
P(-{\hat \zeta}_k^k(\dots,{\hat X}^{(k)}_{-1}, {\hat X}^{(k)}_{0}) 
\ge |{\cal P}_k|^{l_k} 2^{l_k\epsilon})\\
&=& \sum_{y^0_{-l_k+1}\in {\cal P}_k^{l_k}} 
P((\dots, {\hat X}^{(k)}_{-1}, {\hat X}^{(k)}_0, {\hat X}^{(k)}_{1},\dots )\in Q_k(y^0_{-l_k+1}))
\le 
(k+1) 2^{-l_k \epsilon}.
\end{eqnarray*}
By assumption, the right hand side sums, the Borel-Cantelli Lemma   yields that  
$\zeta_k< |{\cal P}_{k}|^{l_k} 2^{l_k\epsilon}$ eventually almost surely
and Theorem~\ref{Theorem4} is proved. 
\end{proof}

\end{document}